\documentclass[12pt,reqno]{amsart}

\usepackage{amssymb}
\usepackage{latexsym}
\usepackage{amsmath}
\usepackage{epsfig}
\usepackage{comment}

\theoremstyle{definition}
\newtheorem{definition}{Definition}[section]

\newtheorem{remark}[definition]{Remark}
\newtheorem{question}[definition]{Question}
\newtheorem{example}[definition]{Example}

\theoremstyle{plain}    
\newtheorem{proposition}[definition]{Proposition}
\newtheorem{lemma}[definition]{Lemma}
\newtheorem{theorem}[definition]{Theorem}
\newtheorem{corollary}[definition]{Corollary}

\newcommand\Cpx{\mathbb{C}}
\newcommand\eps{\epsilon}
\newcommand\distr{\operatorname{distr}}

\newcommand\Nats{\mathbb{N}}
\newcommand\tr{{\mathrm{tr}}}
\newcommand\Tr{{\mathrm{Tr}}}

\newcommand\ev{{\mathrm{ev}}}

\newcount\theTime
\newcount\theHour
\newcount\theMinute
\newcount\theMinuteTens
\newcount\theScratch
\theTime=\number\time
\theHour=\theTime
\divide\theHour by 60
\theScratch=\theHour
\multiply\theScratch by 60
\theMinute=\theTime
\advance\theMinute by -\theScratch
\theMinuteTens=\theMinute
\divide\theMinuteTens by 10
\theScratch=\theMinuteTens
\multiply\theScratch by 10
\advance\theMinute by -\theScratch

\def\today{{\number\day\space
\ifcase\month\or
January\or February\or March\or April\or May\or June\or
July\or August\or September\or October\or November\or December\fi
\space\number\year}}

\begin{document}

\title[Connes' embedding problem]{A Linearization of Connes' Embedding Problem}
\author[Collins]{Beno\^\i{}t Collins$^{\dagger}$}
\address{
Department of Mathematics and Statistics, University of Ottawa,
585 King Edward,
Ottawa, ON
K1N 6N5 Canada, and
CNRS, Department of Mathematics, Lyon 1 Claude Bernard University} 
\email{bcollins@uottawa.ca}
\thanks{\footnotesize $^{\dagger}$Research supported in part by NSERC
grant RGPIN/341303-2007}

\author[Dykema]{Ken Dykema$^{*}$}
\address{Department of Mathematics, Texas A\&M University,
College Station, TX 77843-3368, USA}
\email{kdykema@math.tamu.edu}
\thanks{\footnotesize $^{*}$Research supported in part by NSF grant DMS-0600814}

\date{\today}

\subjclass[2000]{46L10,15A42}

\begin{abstract}
We show that Connes' embedding problem for II$_1$--factors is equivalent to a statement
about distributions of sums of self--adjoint operators with matrix coefficients.
This is an application of a linearization result for finite von Neumann algebras,
which is proved using asymptotic second order freeness of Gaussian random matrices.
\end{abstract}

\maketitle

\section{Introduction}

A von Neuman algebra $\mathcal{M}$
is said to be {\em finite} if it possesses a normal, faithful,
tracial state $\tau$.
By ``finite von Neumann algebra'' $\mathcal{M}$, we will always mean such an algebra
equipped with a fixed such trace $\tau$.
{\em Connes' embedding problem} asks whether every such $\mathcal{M}$ with a separable predual
can be embedded in an ultrapower $R^\omega$ of the hyperfinite II$_1$--factor $R$
in a trace--preserving way.
This is well known to be equivalent to the question of whether a generating set $X$ for $\mathcal{M}$
has microstates, namely, whether there exist matrices over the complex numbers whose
mixed moments
up to an arbitrary given order approximate those of the elements of $X$
with respect to $\tau$, to within an 
arbitrary given tolerance.
(See section~\ref{sec:application-to-embeddability}
where precise definitions and, for completeness, a proof of this equivalence
are given.)
We will say that $\mathcal{M}$ posseses {\em Connes' embedding property} if it embeds in $R^\omega$.
(It is known that possession of this property does not depend on the choice of faithful trace $\tau$.)

Seen like this, Connes' embedding probem, which is open, is about a fundamental approximation property for
finite von Neumann algebras.
There are several important results, due to E.\ Kirchberg~\cite{K93},
F.\ R\u adulescu~\cite{R99}, \cite{R04}, \cite{R06}, \cite{R} and N.\ Brown~\cite{B04},
that have direct bearing on this problem;
see also G.\ Pisier's paper~\cite{P96} and N.\ Ozawa's survey~\cite{Oz04}.

Recently, H.\ Bercovici and W.S.\ Li~\cite{BL06} have proved a property
enjoyed by elements in a finite von Neumann algebra that embeds in $R^\omega$.
This property is related to a fundamental question about spectra of sums of operators:
given Hermitian matrices or, more generally, Hermitian operators $A$ and $B$ with specified spectra, what can
the spectrum of $A+B$ be?
For $N\times N$ matrices, a description was conjectured by Horn~\cite{Horn} and was eventually
proved to be true by work of Klyachko, Totaro, Knutson, Tao and others,
if by ``spectrum'' we mean the {\em eigenvalue sequence}, namely,
the list of eigenvalues repeated according to multiplicity and in non--increasing order.
In this description, the possible spectrum of $A+B$ is a convex subset of $\mathbb{R}^N$ described by certain inequalities,
called the {\em Horn inequalities}.
See Fulton's exposition \cite{F00} or, for a very abbreviated decription,
section~\ref{sec:QHB} of this paper.
We will call this convex set the {\em Horn body} associated to $A$ and $B$,
and denote it by $S_{\alpha,\beta}$, where $\alpha$ and $\beta$ are the eigenvalue sequences of $A$ and $B$, respectively.

Bercovici and Li \cite{BL01}, \cite{BL06} have studied
the analogous question for $A$ and $B$ self--adjoint
elements of a finite von Neumann algebra $\mathcal{M}$,
namely: if spectral data of $A$ and of $B$ are specified, what are the possible
spectral data of $A+B$?
Here, by ``spectral data'' one can take the distribution (i.e., trace of spectral measure)
of the operator in question, which is a compactly supported Borel
probability measure on $\mathbb{R}$, or, in a description that is equivalent,
the {\em eigenvalue function} of the operator, which is a nonincreasing, right--continuous function
on $[0,1)$ that is the non--discrete version of the eigenvalue sequence.

In~\cite{BL06}, for given eigenvalue functions $u$ and $v$,
they construct a
convex set, which we will call $F_{u,v}$, of eigenvalue functions.
This set can be viewed as a limit (in the appropriate sense) of Horn bodies as $N\to\infty$.
They show that
the eigenvalue function of $A+B$ must lie in $F_{u,v}$ whenever
$A$ and $B$ lie in $R^\omega$ and have eigenvalue functions $u$ and, respectively,
$v$.

Bercovici and Li's result provides a concrete method to attempt to show that a finite von Neumann algebra
$\mathcal{M}$ does not embed in $R^\omega$:
find self--adjoint $A$ and $B$ in $\mathcal{M}$ for which one knows enough about the spectral data
of $A$, $B$ and $A+B$, and find a Horn inequality (or, rather, it's appropriate modification
to the setting of eigenvalue functions) that is violated by these.

Their result also inspires two further questions:

\begin{question}
\renewcommand{\labelenumi}{(\roman{enumi})}
\begin{enumerate}
\item Which Horn inequalitites must be satisfied by the spectral data
of self--adjoints $A$, $B$ and $A+B$ in {\em arbitrary} finite von Neumann algebras?

\item (conversely to Bercovici and Li's result):
If we know,
for all self--adjoints $A$ and $B$
in an arbitrary finite von Neumann algebra $\mathcal{M}$,
calling their eigenvalue functions $u$ and $v$, respectively,
that the eigenvalue function of $A+B$ belongs to $F_{u,v}$,
is this equivalent to a positive answer for Connes' embedding problem?
\end{enumerate}
\end{question}

Question (ii) above is easily seen to be equivalent to the same question, but where $A$ and $B$ are assumed
to lie in some copies of the matrix algebra $\mathbb{M}_N(\mathbb{C})$ in $\mathcal{M}$, for some $N\in\mathbb{N}$.

Bercovici and Li, in~\cite{BL01}, partially answered the first question by
showing that a subset of the Horn inequalities (namely, the Freede--Thompson
inequalities) are always satisfied in arbitrary finite von Neuman algebras.

We attempted to address the second question.
We are not able to answer it, but we prove a related result (Theorem~\ref{thm:CEP})
which answers the analogous question for what we call the {\em quantum Horn bodies}.
These are the like the Horn bodies, but with matrix coefficients.
More precisely, if $\alpha$ and $\beta$ are nonincreasing real sequences of length $N$
and if $a_1$ and $a_2$ are self--adjoint $n\times n$ matrices for some $n$,
then the quantum Horn body $K_{\alpha,\beta}^{a_1,a_2}$ is the set of all possible
eigenvalue functions of matrices of the form
\begin{equation}
a_1\otimes U\mathrm{diag}(\alpha)U^*+a_2\otimes V\mathrm{diag}(\beta)V^*
\end{equation}
as $U$ and $V$ range over the $N\times N$ unitaries.
(In fact, Theorem~\ref{thm:CEP} concerns the appropriate union of such bodies over all $N$ ---
see section~\ref{sec:QHB} for details.)

\medskip

Our proof of Theorem~\ref{thm:CEP} is an application of a linearization result (Theorem~\ref{main-1})
in finite von Neumann algebras, which implies that if $X_1$, $X_2$, $Y_1$ and $Y_2$ are
self--adjoint elements of a finite von Neuman algebra and if the distributions (i.e., the moments)
of
\begin{equation}\label{eq:aX}
a_1\otimes X_1+a_2\otimes X_2
\end{equation}
and
\begin{equation}\label{eq:aY}
a_1\otimes Y_1+a_2\otimes Y_2
\end{equation}
agree for all $n\in\Nats$
and all self--adjoint $a_1,a_2\in\mathbb{M}_n(\mathbb{C})$, then the mixed moments of the pair $(X_1,X_2)$
agree with the mixed moments of the pair $(Y_1,Y_2)$, i.e.\ the trace of
\begin{equation}\label{eq:Xword}
X_{i_1}X_{i_2}\cdots X_{i_k}
\end{equation}
agrees with the trace of.
\begin{equation}\label{eq:Yword}
Y_{i_1}Y_{i_2}\cdots Y_{i_k}
\end{equation}
for all $k\in\mathbb{N}$ and all $i_1,\ldots,i_k\in\{1,2\}$.
This is equivalent to there being a trace--preserving isomorphism from the von Neumann algebra generated by $X_1$
and $X_2$ onto the von Neumann algebra generated by $Y_1$ and $Y_2$, that sends $X_i$ to $Y_i$.

This linearization result for von Neumann algebras is quite analogous to one for C$^*$--algebras
proved by U.\ Haagerup and S.\ Thorbj\o{}rnsen~\cite{HT05} (and quoted below as Theorem~\ref{HT-trick}).
However, our proof of Theorem~\ref{main-1} is quite different from that of Haagerup and Thorbj\o{}rnsen's result.
Our linearization result is not so surprising because, for example, for a proof it would suffice
to show that the trace of an arbitrary word of the form~\eqref{eq:Xword}
is a linear combination of moments of various elements of the form~\eqref{eq:aX}.
One could imagine that a combinatorial proof by explicit choice of some $a_1$ and $a_2$, etc., may be possible.
However, our proof does not yield an explicit choice.
Rather, it makes a random choice of $a_1$ and $a_2$.
For this we make crucial use of J.\ Mingo and R.\ Speicher's results on second order freeness of
independent GUE random matrices.

\smallskip

Finally, we need more than just the linearization result.
We use some ultrapower techniques to reverse quantifiers.
In particular, we show that for the von Neumann algebra generated by $X_1$ and $X_2$
to be embeddable in $R^\omega$, it suffices that for all self--adjoint matrices $a_1$ and $a_2$,
there exists $Y_1$ and $Y_2$ lying in $R^\omega$ such that the distributions of~\eqref{eq:aX} and~\eqref{eq:aY}
agree.
For this, it is for technical reasons necessary 
to strengten the linearization result (Theorem~\ref{main-1}) by restricting the matrices $a_1$ and $a_2$
to have spectra in a nontrivial bounded interval $[c,d]$.

\medskip

To recap:
in Section~\ref{sec:linearization} we prove the linearization result, making use of second order freeness.
In Section~\ref{sec:application-to-embeddability}, we review Connes' embedding problem and it's formulation in terms of
microstates;
then we make an ultrapower argument to prove a result (Theorem~\ref{corollaire}) characterizing embeddability
of a von Neumann algebra generated by self--adjoints $X_1$ and $X_2$ in terms of
distributions of elements of the form~\eqref{eq:aX}.
In Section~\ref{sec:QHB}, we describe the quantum Horn bodies, state some related questions and consider some examples.
We finish by rephrasing Connes' embedding problem in terms of the quantum Horn bodies.

\section{Linearization}\label{sec:linearization}

Notation:  we let $\mathbb{M}_n(\mathbb{C})$ denote the set of $n\times n$ complex matrices,
while $\mathbb{M}_n(\mathbb{C})_{s.a.}$
means the set of self--adjoint elements of $\mathbb{M}_n(\mathbb{C})$.
We denote by $\Tr:\mathbb{M}_n(\mathbb{C})\to\mathbb{C}$ the unnormalized trace,
and we let $\tr=\frac1n\Tr$ be the normalized trace (sending the identity element to $1$).

The main theorem of this section is

\begin{theorem}\label{main-1}
Let $\mathcal{M}$ be a von Neumann algebra generated by selfadjoint elements $X_1,\ldots ,X_k$ and
$\mathcal{N}$ be a von Neumann algebra generated by selfadjoint elements $Y_1,\ldots,Y_k$.
Let $\tau$ be a faithful trace on $\mathcal{M}$ and $\chi$ be a faithful trace on $\mathcal{N}$.

Let $c<d$ be real numbers and suppose that for all 
$n\in\mathbb{N}$ and all $a_1,\ldots,a_k$ in $\mathbb{M}_n(\mathbb{C})_{s.a.}$
whose spectra are contained in the interval $[c,d]$,
the distributions of $\sum_i a_i\otimes X_i$ and $\sum_i a_i\otimes Y_i$ are the same.

Then there exists an isomorphism
$\phi:\mathcal{M}\to\mathcal{N}$ such that $\phi (X_i)=Y_i$ and $\chi\circ \phi = \tau$.
\end{theorem}

The statement of this theorem can be thought of as a version for finite von Neumann algebras of the 
following $C^*$--algebra linearization result of Haagerup and Thorbj\o{}rnsen.

\begin{theorem}[\cite{HT05}]\label{HT-trick}
Let $A$ (respectively $B$) be a unital $C^*$--algebra generated by selfadjoints $X_1,\ldots ,X_k$ 
(resp.\ $Y_1,\ldots ,Y_k$) such that for all positive integers $n$ and for all 
$a_0,\ldots ,a_k\in\mathbb{M}_n(\mathbb{C})_{sa}$,
\begin{equation}
a_0\otimes 1+a_1\otimes X_1+\ldots +a_{k}\otimes X_k
\end{equation}
and 
\begin{equation}
a_0\otimes 1+a_1\otimes Y_1+\ldots +a_{k}\otimes Y_k
\end{equation}
have the same spectrum, then
there exists an isomorphism $\phi$ from $A$ onto $B$ such that $\phi (X_i)=Y_i$.
\end{theorem}

However, our proof of Theorem
\ref{main-1} is quite different from the proof of
Theorem \ref{HT-trick}.
In addition, there is the notable difference that
we do not need to consider matrix coefficients of the identity.
In order to simplify our notation, we restrict to proving the $k=2$ case of Theorem \ref{main-1}.
We indicate at Remark~\ref{rem:k} how our proof works in general.

\smallskip

Let $X^{\sharp}$ be the free monoid generated by free elements $x_1,x_2$, and
\begin{equation}
\mathbb{C}\langle x_1,x_2\rangle=\mathbb{C}[X^{\sharp}].
\end{equation}
be the free unital $*$--algebra over selfadjoint elements $x_1,x_2$.

Let $\rho$ be the rotation action of the integers on the set $X^{\sharp}$, given by
\begin{equation}
\rho(x_{i_1}\ldots x_{i_n})=x_{i_2}\ldots x_{i_n}x_{i_1}.
\end{equation}
Let $X^\sharp/\rho$ denote the set of orbits of this action.
Let $\mathcal{I}$ be the vector space spanned by the commutators
$[P,Q]$ with $P,Q\in \mathbb{C}\langle x_1,x_2\rangle$.
Recall that an (algebraic) trace is a linear map $\tau:\mathbb{C}\langle x_1,x_2\rangle\to \mathbb{C}$
such that $\tau (ab)=\tau(ba)$.
Equivalently, a linear map $\tau:\mathbb{C}\langle x_1,x_2\rangle\to \mathbb{C}$
is a trace if and only if it vanishes on $\mathcal{I}$.

\begin{lemma}\label{lem:V}
For any orbit $O\in X^\sharp/\rho$, let $V_O=\mathrm{span}\,O\subseteq\mathbb{C}\langle x_1,x_2\rangle$.
Then $\mathbb{C}\langle x_1,x_2\rangle$ splits as the direct sum
\begin{equation}\label{eq:sumVO}
\mathbb{C}\langle x_1,x_2\rangle=\bigoplus_{O\in X^\sharp/\rho}V_O.
\end{equation}
Moreover, the commutator subspace $\mathcal{I}$ splits accross this direct sum as
\begin{equation}\label{eq:IVO}
\mathcal{I}=\bigoplus_{O\in X^\sharp/\rho}V_O\cap\mathcal{I}.
\end{equation}
Furthermore, $V_O\cap\mathcal{I}$ is of codimension $1$ in $V_O$ and we have
\begin{equation}
V_O\cap\mathcal{I}=\{\sum_{x\in O}c_x x\mid c_x\in\mathbb{C},\,\sum_{x\in O}c_x=0\}.
\end{equation}
\end{lemma}
\begin{proof}
The direct sum decomposition~\eqref{eq:sumVO} is obvious.
From the relation
\begin{equation}
x_{i_1}x_{i_2}\ldots x_{i_n}=
[x_{i_1}x_{i_2}\ldots x_{i_{n-1}},x_{i_n}]
 +x_{i_n}x_{i_1}x_{i_2}\ldots x_{i_{n-1}}\,,
\end{equation}
one easily sees
\begin{gather}
\mathcal{I}\subseteq\{\sum_{x\in O}c_x x\mid c_x\in\mathbb{C},\,\sum_{x\in X^\sharp}c_x=0\} \\
\{\sum_{x\in O}c_x x\mid c_x\in\mathbb{C},\,\sum_{x\in 0}c_x=0\}\subseteq V_O\cap\mathcal{I},
\end{gather}
from which the assertions follow.
\end{proof}

An orbit $O\in X^\sharp/\rho$
is a singleton if and only if it is of the form $\{x_i^a\}$ for some $i\in\{1,2\}$ and some integer $a\ge0$.
For each orbit that is not a singleton, choose a representative of the orbit of the form
\begin{equation}\label{eq:xia}
x=x_1^{a_1}x_2^{b_2}\cdots x_1^{a_n}x_2^{b_n}
\end{equation}
with $n\ge1$ and $a_1,\ldots,a_n,b_1,\ldots,b_n\ge1$,
and collect them together in a set $S$, of representatives for all the orbits in $X^\sharp/\rho$ that are not singletons. 

Let $\tilde U_i$ and $\tilde T_i$ ($i\in\mathbb{N}$) be two families of polynomials,
which we will specify
later on, such
that the degree of each $\tilde U_i$ and $\tilde T_i$ is $i$.
For $x\in S$ written as in~\eqref{eq:xia}, we let
\begin{equation}
\tilde U^x=\tilde U_{a_1}(x_1)\tilde U_{b_1}(x_2)\cdots\tilde U_{a_n}(x_1)\tilde U_{b_n}(x_2)
\in\mathbb{C}\langle x_1,x_2\rangle.
\end{equation}

\begin{lemma}\label{lem:Ponto}
The family
\begin{equation}\label{eq:Xi}
\Xi=\{1\}\cup\{\tilde T_a(x_i)\mid a\in\mathbb{N},\,i\in\{1,2\}\}
\cup\{\tilde U^x\mid x\in S\}\subseteq\mathbb{C}\langle x_1,x_2\rangle 
\end{equation}
is linearly independent
and spans a space $\mathcal{J}$ such that
\begin{align}
\mathcal{I}+\mathcal{J}&=\mathbb{C}\langle x_1,x_2\rangle \label{eq:IJ} \\
\mathcal{I}\cap\mathcal{J}&=\{0\}. \label{eq:IJ0}
\end{align}
\end{lemma}
\begin{proof}
For an orbit $O\in X^\sharp/\rho$, the total degree of all $x\in O$ agree;
denote this integer by $\mathrm{deg}(O)$.
Letting $V_O=\mathrm{span}\,O$ and using Lemma~\ref{lem:V}, an argument by induction
on $\deg(O)$ shows $V_O\subseteq\mathcal{I}+\mathcal{J}$.
This implies~\eqref{eq:IJ}.

To see the linear independece of~\eqref{eq:Xi} and to see~\eqref{eq:IJ0}, suppose
\begin{equation}\label{eq:y}
y=c_01+\sum_{n=1}^\infty(c^{(1)}_a\tilde T_a(x_1)+c^{(2)}_a\tilde T_a(x_2))+\sum_{x\in S}d_x\tilde U^x,
\end{equation}
for complex numbers $c_0$, $c^{(i)}_n$ and $d_x$, not all zero, and let us show $y\notin\mathcal{I}$.
We also write
\begin{equation}
y=\sum_{z\in X^\sharp}a_zz
\end{equation}
for complex numbers $a_z$.

Suppose $d_x\ne0$ for some $x$ and let $x\in S$ be of largest degree such that $d_x\ne0$.
Let $O\in X^\sharp/\rho$ be the orbit of $x$.
Then
\begin{equation}
\sum_{z\in O}a_zz=d_xx\notin V_O\cap\mathcal{I}.
\end{equation}
By the direct sum decomposition~\eqref{eq:IVO}, we get $y\notin\mathcal{I}$.

On the other hand, if $c^{(i)}_n\ne0$ for some $i\in\{1,2\}$ and some $n\ge1$.
Suppose $n$ is the largest such that $c^{(i)}_n\ne0$.
Then $a_{x_i^n}=c^{(i)}_n\ne0$, and $y\notin\mathcal{I}$.

Finally, if $d_x=0$ for all $x\in S$ and
if $c^{(i)}_n=00$ for some $i\in\{1,2\}$ and some $n\ge1$,
then we are left with $c_0\ne1$ and $y=c_01\notin\mathcal{I}$.
\end{proof}

We recall that a Gaussian unitary ensemble (also denoted by
{\em GUE}) is the probability distribution of the random matrix $Z_N+Z_N^*$ 
on $\mathbb{M}_N(\mathbb{C})$,
where $Z_N$ has independent complex 
gaussian entries of variance $1/2N$.
This distribution has a density proportional to  $e^{-N\Tr X^2}$ with respect to the Lebesgue measure
on the selfadjoint real matrices. 
A classical result of Wigner~\cite{W55} states that
the empirical eigenvalue distribution of a {\em GUE} converges as $N\to\infty$ in moments to
Wigner's semi--circle distribution
\begin{equation}
\frac{1}{2\pi}1_{[-2,2]}(x)\sqrt{4-x^2}dx.
\end{equation}

If we view the $X_N$ for various $N$ as matrix--valued random variables over a commone probability space,
then almost surely, the largest and smallest eigenvalues of $X_N$ converge as $N\to\infty$
to $\pm2$, respectively.
This was proved by Bai and Yin~\cite{BY88}
(see also~\cite{B99}).
See~\cite{HT03} for further discussion and an alternative proof.

We recall that the Chebyshev polynomials of the first kind
$T_i$ are the monic polynomials
orthogonal with respect to the weight $1_{(-2,2)}(x)(4-x^2)^{-1/2}dx$.
Alternatively, they are determined by their generating series
\begin{equation}
\sum_{i\geq 0} T_i(x)t^i=\frac{1-tx}{1-2tx+t^2}
\end{equation}

Similarly,  Chebyshev polynomial of the second kind $U_i$
are orthogonal with respect to the weight 
$1_{[-2,2]}(x)(4-x^2)^{1/2}dx$ and have the generating series
\begin{equation}
\sum_{i\geq 0} T_i(x)t^i=\frac{1}{1-2tx+t^2}
\end{equation}

The following result is random matrix folklore, but it is implied by more general
results of  Johansson (\cite{J98}, Cor 2.8):

\begin{proposition}\label{prop:johansson}
Let $X_N$ be the
{\em GUE} of dimension $N$ and $T_n$ the
Chebyshev polynomial of second kind.
Let
\begin{equation}\label{eq:alphan}
\alpha_n=\frac1{2\pi}\int_{-2}^2T_n(t)\sqrt{4-t^2}\,dt.
\end{equation}
Then for every $m\in\mathbb{N}$, the real random vector
\begin{equation}
2\bigg(\frac{\Tr (T_n(X_N))-N\alpha_n}{\sqrt{n}}\bigg)_{n=1}^m
\end{equation} 
tends in distribution as $N\to\infty$ toward a vector of independent standard real Gaussian variables.
\end{proposition}

Consider two GUE random
matrix ensembles $(X_N)_{N\in\mathbb{N}}$ and 
$(Y_N)_{N\in\mathbb{N}}$, that are independent from each other (for each $N$).
Voiculescu proved~\cite{V91} that these converge in moments to free semicircular elements $s_1$ and $s_2$
having first moment zero and second moment $1$,
meaning that we have
\begin{equation}
\lim_{N\to\infty}E(\mathrm{tr}(X_N^{k_1}Y_N^{\ell_1}\cdots X_N^{k_m}Y_N^{\ell_m}))
=\tau(s_1^{k_1}s_2^{\ell_2}\cdots s_1^{k_m}s_2^{\ell_m})
\end{equation}
for all $m\ge1$ and $k_i,\ell_i\ge0$, (where $\tau$ is a trace with respect to which $s_1$ and $s_2$ are semicircular
and free).
Of course, by freeness, this implies that if $p_i$ and $q_i$ are polynomials
such that $\tau(p_i(s_1))=0=\tau(q_i(s_2))$
for all $i\in\{1,\ldots,m\}$, then
\begin{equation}
\lim_{N\to\infty}E(\mathrm{tr}(p_1(X_N)q_1(Y_N)\cdots p_m(X_N)q_m(Y_N)))=0.
\end{equation}
Mingo and Speicher~\cite{MS06}
have proved some remarkable results about the related fluctuations, namely, the (magnified) random variables~\eqref{eq:mixedfluc}
below.
These are asymptotically Gaussian and provide examples of the phenomenon of second order freeness, which has been
treated in a recent series of papers~\cite{MS06}, \cite{MSS07}, \cite{CMSS07}.
In particular, the following theorem is a straightforward consequence of some of the results in~\cite{MS06}.

\begin{theorem}\label{MS}
Let $X_N$ and $Y_N$ be independent GUE random matrix ensembles.
Let $s$ be a $(0,1)$--semicircular element with respect to a trace $\tau$.
Let $m\ge1$ and let $p_1,\ldots,p_m,q_1,\ldots,q_m$ be polynomials with real coefficients such that
$\tau(p_i(s))=\tau(q_i(s))=0$
for each $i$.
Then the random variable
\begin{equation}\label{eq:mixedfluc}
\mathrm{Tr}(p_1(X_N)q_1(Y_N)\cdots p_m(X_N)q_m(Y_N))
\end{equation}
converges in moments as $N\to\infty$ to a Gaussian random variable.
Moreover, if $\tilde m\ge1$ and if $\tilde p_1,\ldots,\tilde p_{\tilde m},\tilde q_1,\ldots,\tilde q_{\tilde m}$
are real polynomials such that 
$\tau(\tilde p_i(s))=\tau(\tilde q_i(s))=0$
for each $i$, then
\begin{align}
\lim_{N\to\infty}E\big(&\mathrm{Tr}(p_1(X_N)q_1(Y_N)\cdots p_m(X_N)q_m(Y_N))\cdot \\
&\quad\overline{\mathrm{Tr}(\tilde p_1(X_N)\tilde q_1(Y_N)\cdots \tilde p_{\tilde m}(X_N)\tilde q_{\tilde m}(Y_N))}\,\big) \\
&=\begin{cases}
\sum_{\ell=0}^{m-1}\prod_{j=1}^m\tau(p_j(s)\tilde p_{j+\ell}(s))\tau(q_j(s)\tilde q_{j+\ell}(s)),&m=\tilde m \\
0,&m\ne\tilde m,\end{cases}
\end{align}
where the subscripts of $p$ and $q$ are taken modulo $m$.
Furthermore, for any polynomial $r$, we have
\begin{align}
\lim_{N\to\infty}E\big(\mathrm{Tr}(p_1(X_N)q_1(Y_N)\cdots p_m(X_N)q_m(Y_N))\Tr(r(X_N))\big)&=0 \\
\lim_{N\to\infty}E\big(\mathrm{Tr}(p_1(X_N)q_1(Y_N)\cdots p_m(X_N)q_m(Y_N))\Tr(r(Y_N))\big)&=0.
\end{align}
\end{theorem}

If $\mathfrak{A}$ is any unital algebra and if $a_1,a_2\in\mathfrak{A}$, we let 
\begin{equation}
\ev_{a_1,a_2}:\mathbb{C}\langle x_1,x_2\rangle\to\mathfrak{A}
\end{equation}
be the algebra homomorphism given by
\begin{equation}
\ev_{a_1,a_2}(P)=P(a_1,a_2).
\end{equation}
In the corollary below, which follows directly from
Theorem \ref{MS} and Proposition \ref{prop:johansson},
we take as $\mathfrak{A}$ the algebra of random matrices (over a fixed probability space)
whose entries have moments of all orders.

\begin{corollary}\label{cor:MS}
Let $u$ and $v$ be real numbers with $u<v$.
Let $A_{N},B_{N}$ be independent
copies of 
\begin{equation}
\frac{u+v}{2}Id+\frac{v-u}{2}X
\end{equation} 
where $X$ is distributed as the
{\em GUE} of dimension $N$.
Let 
\begin{equation}
\tilde T_i(x):=T_i(\frac2{v-u}x-\frac{u+v}{v-u}).
\end{equation}
and
\begin{equation}
\tilde U_i(x):=U_i(\frac2{v-u}x-\frac{u+v}{v-u}).
\end{equation}
If $y\in S$, then we have
\begin{equation}
\lim_{N\to\infty}E(\tr\,\circ\,\ev_{A_N,B_N}(y))=0,
\end{equation}
and we let $\beta(y)=0$.
If $y=x_i^n$ for $i\in\{1,2\}$ and $n\in\Nats$, then we have
\begin{equation}
\lim_{N\to\infty}E(\tr\,\circ\,\ev_{A_N,B_N}(y))=\alpha_n\,,
\end{equation}
where $\alpha_n$ is as in~\eqref{eq:alphan}, and we set $\beta(y)=\alpha_n$.

Then the random variables
\begin{equation}
\left(\; (\Tr\,\circ\,\ev_{A_N,B_N})(y)-N\beta(y)\;    \right)_{y\in \Xi\backslash\{1\}}\,,
\end{equation} 
where $\Xi$ is as in Lemma \ref{lem:Ponto},
converge in moments as $N\to\infty$
to independent, non--trivial, centered, Gaussian variables.
\end{corollary}

The following lemma is elementary and
we will only use it in the especially simple
case of $\delta=0$.
We will use it to see
that for a sequence $z_N$ of random variables converging in moments
to a nonzero random variable, we have that
$\mathrm{Prob}(z_N\ne0)$ is bounded away from zero as $N\to\infty$.
This is all unsurprising and well known, but we include proofs for completeness.

\begin{lemma}\label{lem:elem}
Let $y$ be a random variable with finite first and second moments, denoted
$m_1$ and $m_2$.
Suppose $y\ge0$ and $m_1>0$.
Then for every $\delta>0$ satisfying
\begin{equation}
0\le\delta<\min(\frac{m_2}{2m_1},m_1),
\end{equation}
there is $w$, a continuous function of $m_1$, $m_2$ and $\delta$, such that $0\le w<1$ and
\begin{equation}
\mathrm{Prob}(y\le\delta)\le w.
\end{equation}
More precisely, we may choose
\begin{equation}\label{eq:wform}
w=\begin{cases}\frac{-m_2+2\delta m_1+\sqrt{m_2^2-4\delta m_2(m_1-\delta)}}{2\delta^2},&\delta>0, \\
1-\frac{m_1^2}{m_2},&\delta=0.
\end{cases}
\end{equation}
\end{lemma}
\begin{proof}
Say that $y$ is a random variable on a probability space $(\Omega,\mu)$ and let $V\subseteq\Omega$
be the set where $y$ takes values $\le\delta$.
Using the Cauchy--Schwarz inequality, we get
\begin{equation}
m_1\le\delta\mu(V)+\int_{V^c}y\,d\mu
\le\delta\mu(V)+m_2^{1/2}(1-\mu(V))^{1/2},
\end{equation}
which yields
\begin{equation}
\delta^2\mu(V)^2+(m_2-2\delta m_1)\mu(V)+m_1^2-m_2\le0.
\end{equation}
If $\delta=0$, then this gives
$\mu(V)\le 1-\frac{m_1^2}{m_2}=:w$.
When $\delta>0$,
consider the polynomial
\begin{equation}
p(x)=\delta^2x^2+(m_2-2\delta m_1)x+m_1^2-m_2.
\end{equation}
It's minimum value occurs at $x=\frac{2\delta m_1-m_2}{2\delta^2}<0$
and we have $p(0)=m_1^2-m_2\le0$ (by the Cauchy--Schwarz inequality) and $p(1)=(\delta-m_1)^2>0$.
Therefore, letting $r_2$ denote the larger of the roots of $p$, we have $0\le r_2<1$.
Moreover, if $x\ge0$ and $p(x)\le0$, then $x\le r_2$.
Taking $w=r_2$, we conclude that $\mu(V)\le w$, and we have the formula~\eqref{eq:wform}.
It is easy to see that $w$ is a continuous function of $m_1$, $m_2$ and $\delta$.
\end{proof}

\begin{lemma}\label{lem:stepII}
Let $c<d$ be real numbers.
For matrices $a_1$ and $a_2$, consider the maps $\Tr\,\circ\,\ev_{a_1,a_2}:\mathbb{C}\langle x_1,x_2\rangle\to\mathbb{C}$.
Then we have
\begin{equation}\label{eq:kers}
\bigcap_{\substack{
N\in\Nats \\
a_1,a_2\in \mathbb{M}_N(\mathbb{C})_{sa} \\
c1\le a_i\le d1,\,(i=1,2)}}
\ker(\Tr\,\circ\,\ev_{a_1,a_2})
=\mathcal{I}.
\end{equation}
\end{lemma}
\begin{proof}
The inclusion $\supseteq$ in~\eqref{eq:kers} follows from the trace property.

Let $c<u<v<d$ and
make the choice of polynomials $\tilde T_i$ and $\tilde U_i$ described in Corollary~\ref{cor:MS}.
Letting $\Xi$ and $\mathcal{J}$ be as in Lemm~\ref{lem:Ponto},
for each $y\in\mathcal{J}\backslash\{0\}$,
we will find matrices $a_1$ and $a_2$ such that
\begin{equation}\label{eq:Tr0}
\Tr(\ev_{a_1,a_2}(y))\ne0.
\end{equation}
By~\eqref{eq:IJ} and~\eqref{eq:IJ0} of Lemma~\ref{lem:Ponto}, this will suffice to show $\subseteq$ in~\eqref{eq:kers}.
Rather than find $a_1$ and $a_2$ explicitly,
we make use of random matrices.

We may write
\begin{equation}
y=c_01+\sum_{n=1}^\infty(c^{(1)}_a\tilde T_a(x_1)+c^{(2)}_a\tilde T_a(x_2))+\sum_{x\in S}d_x\tilde U^x,
\end{equation}
with $c_0$, $c^{(i)}_n$ and $d_x$, not all zero.
If $c_0$ is the only nonzero coefficient,  then $y$ is a nonzero constant multiple of the identity and
any choice of $a_1$ and $a_2$ gives~\eqref{eq:Tr0}.
So assume some $c^{(i)}_n\ne0$ or $d_x\ne0$.
Let $A_N$ and $B_N$ be the independent $N\times N$ random matrices as described in Corollary~\ref{cor:MS}.
Extend the function $\beta:\Xi\backslash\{1\}\to\mathbb{R}$ that was defined in Corollary~\ref{cor:MS} to
a function $\beta:\mathcal{J}\to\mathbb{R}$ by linearity
and by setting $\beta(1)=1$.
By that corollary, the random variable
\begin{equation}
z_N:=\Tr\,\circ\,\ev_{A_N,B_N}(y)-N\beta(y)
\end{equation}
converges as $N\to\infty$ in moments to a Gausian random variable
with some nonzero variance $\sigma^2$.
It is now straightforward to see that
\begin{equation}\label{eq:Trev}
\mathrm{Prob}(\Tr\,\circ\,\ev_{A_N,B_N}(y)\ne0)
\end{equation}
is bounded away from zero as $N\to\infty$.
Indeed,
If $\beta(y)\ne0$, then since $N\beta(y)\to\pm\infty$ and since the second moment of $z_N$ stays bounded
as $N\to\infty$, the quantity~\eqref{eq:Trev} stays bounded away from zero as $N\to\infty$.
On the other hand, if $\beta(y)=0$, then
considering the second and fourth moments of $z_N$ and applying Lemma~\ref{lem:elem}, we find $w<1$ such that
for all $N$ sufficiently large, we have
$\mathrm{Prob}(z_N\ne0)\ge1-w$.
Thus, also in this case, the quantity~\eqref{eq:Trev} is bounded away from zero as $N\to\infty$.

{}By work of Haagerup and Thorbj\o{}rnsen (see equation~(3.7) and the next displayed equation
of~\cite{HT03}), we have
\begin{equation}\label{eq:HT}
\lim_{N\to\infty}\mathrm{Prob}(c1\le A_N\le d1)=1,
\end{equation}
and also for $B_N$.
Combining boundedness away from zero of~\eqref{eq:Trev} with~\eqref{eq:HT},
for some $N$ sufficiently large, we can evaluate $A_N$ and $B_N$ on a set
of nonzero measure to obtain $a_1,a_2\in\mathbb{M}_N(\mathbb{C})$ so that
$\Tr\,\circ\,\ev_{a_1,a_2}(y)\ne0$ and $c1\le a_i\le d1$ for $i=1,2$.
\end{proof}

\begin{proof}[Proof of Theorem \ref{main-1}]
As mentioned before, we concentrate on the case $k=2$,
and the other cases follow similarly.
By the Gelfand--Naimark--Segal
representation theorem, it is enough to prove that for all monomials $P$ in $k$ non--commuting variables,
we have
\begin{equation}
\tau (P(X_i))=\chi (P(Y_i)).
\end{equation}
Rephrased, this amounts to showing that we have
\begin{equation}\label{eq:tauchi}
\tau\circ\ev_{X_1,X_2}(x)=\chi\circ\ev_{Y_1,Y_2}(x)
\end{equation}
for all $x\in X^{\sharp}$.
By hypothesis,
for all $p\geq 0$, all $N\in\mathbb{N}$ and all $a_1,a_2\in\mathbb{M}_N(\mathbb{C})$ we have
\begin{equation}\label{tmp1}
\tr\otimes \tau ((a_1\otimes X_1+a_2\otimes X_2)^p)=\tr\otimes \chi ((a_1\otimes Y_1+a_2\otimes Y_2)^p).
\end{equation}
Developing the right--hand--side  minus the left--hand--side of~\eqref{tmp1}
gives that the equality
\begin{equation}
\sum_{i_1,\ldots i_p\in\{1,2\}} \tr (a_{i_1}\ldots a_{i_p})(\tau (X_{i_1}\ldots X_{i_p}) -\chi (Y_{i_1}\ldots Y_{i_p}))=0
\end{equation}
holds true for any choice $a_1,a_2\in\mathbb{M}_N(\mathbb{C})_{sa}$.
This equation can be rewritten as
\begin{equation}\label{eq:tauchi0}
\sum_{x\in S_p} 
c_x\big((\tr\,\circ\ev_{a_1,a_2})(x)\big)(\tau\circ\ev_{X_1,X_2}(x)-\chi\circ\ev_{Y_1,Y_2}(x))=0,
\end{equation}
where $S_p\subset X^\sharp$ is a set representatives, one from each orbit in $X^\sharp/\rho$,
of the monomials of degree $p$,
and where $c_x$ is the cardinality of each class.

Suppose, for contradiction, that~\eqref{eq:tauchi} fails for some $x\in S_p$.
Let
\begin{equation}
y=\sum_{x\in S_p} 
c_x(\tau\circ\ev_{X_1,X_2}(x)-\chi\circ\ev_{Y_1,Y_2}(x))x\in\mathbb{C}\langle x_1,x_2\rangle.
\end{equation}
By Lemma~\ref{lem:V},
$y\notin\mathcal{I}$.
By Lemma~\ref{lem:stepII}, there are $N\in\Nats$ and $a_1,a_2\in\mathbb{M}_N(\mathbb{C})$
such that $c1\le a_i\le d1$ for $i=1,2$ and $\tr\,\circ\ev_{a_1,a_2}(y)\ne0$.
But $\tr\,\circ\ev_{a_1,a_2}(y)$ is the left--hand--side of~\eqref{eq:tauchi0},
and we have a contradiction.
\end{proof}

\begin{remark}\label{rem:k}
We only proved the result for $k=2$. The proof for arbitrary $k$ is actually exactly the same.
The only difference is that the notations in the definition of second order freeness is more
cumbersome, but Theorem \ref{MS} as well as the other lemmas are unchanged.
\end{remark}

\begin{remark}\label{rem:ay}
The main ingredient in the proof of Theorem \ref{main-1} is to provide a method
of constructing $a_1,a_2\in \mathbb{M}_N(\mathbb{C})_{sa}$ such that
\begin{equation}\label{eq:Tey}
(\Tr\,\circ\,\ev_{a_1,a_2})(y)\ne0,
\end{equation}
whenever this is not ruled out by reasons of symmetry.
Our approach is probabilistic, and makes unexpected use of second--order freeness.
In particular, our approach is non--constructive.
It would be interesting to find a direct approach. 
\end{remark}

It is natural to wonder how much one can shrink the choice of matrices from which $a_1$ and $a_2$
in Remark~\ref{rem:ay} are drawn.
We would like to point out here that
in Lemma~\ref{lem:stepII}
we needs at least infinitely 
many values of $N$.
More precisely, we can prove the following:

\begin{proposition}
For each $N_0\in\mathbb{N}$, we have
\begin{equation}
\bigcap_{\substack{N\leq N_0 \\ a_1,a_2\in \mathbb{M}_N(\mathbb{C})_{sa}}} \ker (\Tr\,\circ\,\ev_{a_1,a_2})\supsetneqq\mathcal{I}.
\end{equation}
\end{proposition}
\begin{proof}
Without loss of generality (for example, by taking $N_0!$), it will be enough to prove
\begin{equation}\label{eq:kerTrev}
\bigcap_{a_1,a_2\in \mathbb{M}_N(\mathbb{C})_{sa}} \ker (\Tr\,\circ\,\ev_{a_1,a_2})\supsetneqq\mathcal{I}
\end{equation}
for each $N\in\mathbb{N}$.

Following the proof of Theorem~\ref{main-1},
let $W_p=\mathrm{span}\,\{x+\mathcal{I}\mid x\in S_p\}$ be the degree $p$ vector subspace of the quotient of vector spaces
$\mathbb{C}\langle x_1,x_2\rangle/\mathcal{I}$.
The dimension of $W_p$ is at least $2^p/p$.

Consider the commutative polynomial algebra $\mathbb{C}[x_{11},\ldots,x_{NN},y_{11},\ldots,y_{NN}]$
in the $2N^2$ variables $\{x_{ij},y_{ij}\mid1\le i,j\le N\}$.
Consider matrices
\begin{equation}
X=(x_{ij}),\;
Y=(y_{ij})\in\mathbb{M}_N(\mathbb{C})\otimes\mathbb{C}[x_{11},\ldots,x_{NN},y_{11},\ldots,y_{NN}]
\end{equation}
over this ring.
In this setting,
\begin{equation}
\phi:=(\Tr\otimes\mathrm{id}_{\mathbb{C}[x_{11},\ldots,x_{NN},y_{11},\ldots,y_{NN}]})\circ \ev_{X,Y}
\end{equation}
is a $\mathbb{C}$-linear map from 
$\mathbb{C}\langle x_1,x_2\rangle$ to $\mathbb{C}[x_{11},\ldots,x_{NN},y_{11},\ldots,y_{NN}]$
that vanishes on $\mathcal{I}$
and every map of the form $\Tr\,\circ\,\ev_{a_1,a_2}$ for $a_1,a_2\in\mathbb{M}_N(\mathbb{C})$
is $\phi$ composed with some evaluation map
on the polynomial ring $\mathbb{C}[x_{11},\ldots,x_{NN},y_{11},\ldots,y_{NN}]$.
Therefore, we have
\begin{equation}\label{eq:kerphi}
\ker\phi\subseteq\bigcap_{a_1,a_2\in \mathbb{M}_N(\mathbb{C})_{sa}} \ker (\Tr\,\circ\,\ev_{a_1,a_2}).
\end{equation}

We denote also by $\phi$ the induced map 
\begin{equation}
\mathbb{C}\langle x_1,x_2\rangle/\mathcal{I}
\to\mathbb{C}[x_{11},\ldots,x_{NN},y_{11},\ldots,y_{NN}].
\end{equation}
Clearly, $\phi$ maps $\mathbb{C}\langle x_1,x_2\rangle_p$ into the vector space of homogeneous polynomials
in $\mathbb{C}[x_{11},\ldots,x_{NN},y_{11},\ldots,y_{NN}]$ of
degree $p$.
The space of homogenous polynomials of degree $p$ in $M$ variables has dimension equal
to the binomial coefficient $\binom{p+M-1}{M-1}$.
Therefore, there exists a constant $C>0$, depending on $N$, such that 
$\phi$ maps into a subspace of complex dimension $\le Cp^{N^2-1}$.
For fixed $N$, there is $p$ large enough so that one has $2^p/p>C p^{N^2-1}$.
Therefore, by the rank theorem, the kernel of $\phi$ restricted to 
$\mathbb{C}\langle x_1,x_2\rangle_p$ must be non--empty.
Combined with~\eqref{eq:kerphi}, this proves~\eqref{eq:kerTrev}.
\end{proof}

\section{Application to embeddability}\label{sec:application-to-embeddability}

We begin by recalling the ultrapower construction.
Let $R$ denote the hyperfinite II$_1$--factor and $\tau_R$ its normalized trace.
Let $\omega$ be a free ultrafilter on $\Nats$ and let $I_\omega$ denote
the ideal of $\ell^\infty(\Nats,R)$ consisting of those sequences $(x_n)_{n=1}^\infty$
such that $\lim_{n\to\omega}\tau_R((x_n)^*x_n)=0$.
Then $R^\omega$ is the quotient $\ell^\infty(\Nats,R)/I_\omega$, which is actually a von Neumann
algebra.

Let $\mathcal{M}$ be a von Neumann algebra with
normal, faithful, tracial state $\tau$. 

\begin{definition}
The von Neumann algebra $\mathcal{M}$ 
is said to have {\em Connes' embedding property}
if $\mathcal{M}$ can be embedded into an ultra power $R^\omega$ of the hyperfinite von Neumann algebra $R$
in a trace--preserving way.
\end{definition}

\begin{definition}
If $X=(x_1,\ldots,x_n)$ is a finite subset of $\mathcal{M}_{sa}:=\{x\in\mathcal{M}\mid x^*=x\}$, we say that $X$ 
{\em has matricial microstates} if for every $m\in\Nats$ and every $\eps>0$,
there is $k\in\Nats$
and there are self--adjoint $k\times k$ matrices $A_1,\ldots,A_n$ such that whenever $1\le p\le m$
and $i_1,\ldots,i_p\in\{1,\ldots,n\}$, we have
\begin{equation}\label{eq:Amomxmom}
|\tr_k(A_{i_1}A_{i_2}\cdots A_{i_p})-\tau(x_{i_1}x_{i_2}\cdots x_{i_p})|<\eps,
\end{equation}
where $\tr_k$ is the normalized trace on $\mathbb{M}_k(\Cpx)$.
\end{definition}

It is not difficult to see that if $X$ has matricial microstates, then for every $m\in\Nats$
and $\eps>0$, there is $K\in\Nats$ such that for every $k\ge K$ there are matrices
$A_1,\ldots,A_n\in\mathbb{M}_k(\Cpx)$ whose mixed moments approximate those of $X$ in the sense
specified above.
Also, as proved by an argument of Voiculescu~\cite{VII},
if $X$ has matricial microstates, then 
each approximating
matrix $A_j$ above can be chosen to have norm
no greater than $\|x_j\|$.

The following result is well known.
For future reference, we briefly describe a proof.
\begin{proposition}\label{prop:microstates}
Let $\mathcal{M}$ be a von Neumann algebra with seperable predual
and $\tau$ a normal, faithful, tracial state on $\mathcal{M}$.
Then the following are equivalent:
\begin{itemize}
\item[(i)] $\mathcal{M}$ has Connes' embedding property.
\item[(ii)] Every finite subset $X\subseteq\mathcal{M}_{sa}$ has matricial microstates.
\item[(iii)] If $Y\subseteq M_{sa}$ is a generating set for $\mathcal{M}$, then
every finite subset $X$ of $Y$ has matricial microstates.
\end{itemize}
In particular, if $Y$ is a finite generating set of $\mathcal{M}$ then the above conditions are equivalent
to $Y$ having matricial microstates.
\end{proposition}
\begin{proof}
The implication (i)$\implies$(ii) follows because if
$X=(x_1,\ldots,x_n)\subseteq(R^\omega)_{sa}$, then
choosing any representatives of the $x_j$ in $\ell^\infty(\Nats,R)$, we find elements
$a_1,\ldots,a_n$ of $R$ whose mixed moments up to order $m$ approximate those of the $x_j$
as closely as desired.
Now we use that any finite subset of $R$ is approximately (in $\|\,\|_2$--norm)
contained in some copy $\mathbb{M}_k(\Cpx)\subseteq R$, for some $k$ sufficiently large.

The implication (ii)$\implies$(iii) is evident.

For (iii)$\implies$(i), we may without loss of generality suppose that $Y=\{x_1,x_2,\ldots\}$
for some sequence $(x_j)_1^\infty$ possibly with repetitions.
Fix $m\in\Nats$, let $k\in\Nats$ and let $A_1^{(m)},\ldots,A_m^{(m)}\in\mathbb{M}_k(\Cpx)$ be
matricial microstates for $x_1,\ldots,x_m$ so that~\eqref{eq:Amomxmom}
holds for all $p\le m$ and for $\eps=1/m$, and assume $\|A_i^{(m)}\|\le\|x_i\|$ for all $i$.
Choose a unital $*$--homomorphism $\pi_k:\mathbb{M}_k(\Cpx)\hookrightarrow R$,
and let $a_i^m=\pi_k(A_i^{(m)})$.
Let $b_i=(a_i^m)_{m=1}^\infty\in\ell^\infty(\Nats,R)$,
where we set $a_i^m=0$ if $i>m$.
Let $z_i$ be the image of $b_i$ in $R^\omega$.
Then $z_1,z_2,\ldots$ has the same joint distribution as $x_1,x_2,\ldots$, and
this yields an embedding $M\hookrightarrow R^\omega$ sending $x_i$ to $z_i$.
\end{proof}

A direct consequence of Theorem~\ref{main-1} is:

\begin{theorem}\label{corollaire}
Suppose that a von Neumann algebra $\mathcal{M}$ with trace $\tau$ is generated by self--adjoint elements
$x_1$ and $x_2$.
Let $c<d$ be real numbers.
Then $\mathcal{M}$ has Connes' embedding property if and only if
there exists $y_1,y_2\in (R^{\omega})_{sa}$
such that for all  $a_1,a_2\in\mathbb{M}_n(\mathbb{C})_{sa}$ whose spectra are contained in $[c,d]$,
\begin{equation} \label{eq:aXaY}
\distr(a_1\otimes x_1+a_2\otimes x_2)=\distr(a_1\otimes y_1+a_2\otimes y_2).
\end{equation}
\end{theorem}

In this section we will prove that Connes' embedding property
is equivalent to a weaker condition.

\begin{lemma}\label{lem:YI}
Suppose that a von Neumann algebra $\mathcal{M}$ with trace $\tau$ is generated by self--adjoint elements
$x_1$ and $x_2$.
Let $c<d$ be real numbers and
for every $n\in\Nats$, let $E_n$ be a dense subset of the set
of all elements of $\mathbb{M}_n(\Cpx)$ whose spectra are contained in the interval $[c,d]$.
Then $\mathcal{M}$ has Connes' embedding property if and only if
for all 
finite sets $I$ and all choices of $n(i)\in I$ and $a_1^i,a_2^i\in E_{n(i)}$,
$(i\in I)$,
there exists $y_1,y_2\in R^{\omega}_{s.a.}$ such that
\begin{gather}
\distr(x_1)=\distr(y_1) \label{eq:x1y1} \\
\distr(x_2)=\distr(y_2) \\
\distr(a_1^i\otimes x_1+a_2^i\otimes x_2)=\distr(a_1^i\otimes y_1+a_2^i\otimes y_2),\quad(i\in I).
\label{eq:aiyaix}
\end{gather}
\end{lemma}
\begin{proof}
Necessity is clear.

For sufficiency, we'll use an ultraproduct argument.
Let $(a_1^i,a_2^i)_{i\in\mathbb{N}}$ be an enumeration of a countable, dense subset of
the disjoint union
$\sqcup_{n\geq 1} E_n\times E_n$.
We let $n(i)$ be such that $a_1^i,a_2^i\in\mathbb{M}_{n(i)}(\Cpx)$.
For each $m\in\Nats$, let
$y_1^m,y_2^m$ be elements of $R^{\omega}$ satisfying $\distr(y_j^m)=\distr(x_j)$ and
\begin{equation}
\distr(a_1^i\otimes x_1+a_2^i\otimes x_2)=\distr(a_1^i\otimes y_1^m+a_2^i\otimes y_2^m)
\end{equation}
for all $i\in \{1,\ldots,m\}$.
In particular,
$\|y_j^m\|=\|x_j\|$ for $j=1,2$ and all $m$.
Let
\begin{equation}
b_j^m=(b_{j,n}^m)_{n=1}^\infty\in\ell^\infty(\Nats,R)
\end{equation}
be such that $\|b_j^m\|\le\|x_j\|+1$ and the image of $b_j^m$ in $R^\omega$ is $y_j^m$ ($j=1,2$).
This implies that for all $p\in\Nats$ and all $i\in\{1,\ldots,m\}$, we have
\begin{equation}
\lim_{k\to\omega}\tr_{n(i)}\otimes\tau_R\big((a_1^i\otimes b_{1,k}^m+a_2^i\otimes b_{2,k}^m)^p\big)
=\tr_{n(i)}\otimes\tau\big((a_1^i\otimes x_1+a_2^i\otimes x_2)^p\big),
\end{equation}
which in turn implies that there is a set $F_m$ belonging to the ultrafilter $\omega$ such that 
for all $p,i\in\{1,\ldots,m\}$ and all $k\in F_m$, we have
\begin{equation}
\big|\tr_{n(i)}\otimes\tau_R\big((a_1^i\otimes b_{1,k}^m+a_2^i\otimes b_{2,k}^m)^p\big)
-\tr_{n(i)}\otimes\tau\big((a_1^i\otimes x_1+a_2^i\otimes x_2)^p\big)|<\frac1m.
\end{equation}
For $q\in\Nats$, let $k(q)\in\cap_{m=1}^q F_m$ and for $j=1,2$, let
\begin{equation}
b_j=(b_{j,k(q)}^q)_{q=1}^\infty\in\ell^\infty(\Nats,R).
\end{equation}
Then for all $i,p\in\Nats$, we have
\begin{equation}
\lim_{q\to\infty}\tr_{n(i)}\otimes\tau_R\big((a_1^i\otimes b_{1,k(q)}^q+a_2^i\otimes b_{2,k(q)}^q)^p\big)
=\tr_{n(i)}\otimes\tau\big((a_1^i\otimes x_1+a_2^i\otimes x_2)^p\big),
\end{equation}
Let $y_j$ be the image in $R^\omega$ of $b_j$.
Then we have
\begin{equation}
\distr(a_1^i\otimes x_1+a_2^i\otimes x_2)=\distr(a_1^i\otimes y_1+a_2^i\otimes y_2)
\end{equation}
for all $i\in\Nats$.
By density, we have that~\eqref{eq:aXaY} holds for all $n\in\Nats$ and all
$a_1,a_2\in\mathbb{M}_n(\Cpx)_{sa}$ having spectra in $[c,d]$.
Therefore, by Theorem~\ref{corollaire}, $\mathcal{M}$ is embeddable in $R^\omega$.
\end{proof}

\begin{theorem}\label{switch-indices}
Suppose that a von Neumann algebra $\mathcal{M}$ with trace $\tau$ is generated by self--adjoint elements
$x_1$ and $x_2$ and suppose that both $x_1$ and $x_2$ are positive and invertible.
Then $\mathcal{M}$ has Connes' embedding property if and only if
for all $n\in\Nats$ and all  $a_1,a_2\in\mathbb{M}_n(\mathbb{C})_+$ 
there exists $y_1,y_2\in R^{\omega}_{s.a.}$ such that
\begin{gather}
\distr(x_1)=\distr(y_1) \label{eq:xy1} \\
\distr(x_2)=\distr(y_2) \label{eq:xy2} \\
\distr(a_1\otimes x_1+a_2\otimes x_2)=\distr(a_1\otimes y_1+a_2\otimes y_2) \label{eq:axay}
\end{gather}
hold.
\end{theorem}
\begin{proof}
Again, necessity is clear.

For the reverse implication, we will show that the conditions of Lemma~\ref{lem:YI}
are satisfied.
Suppose that for all $n\in\Nats$ and all $a_1,a_2\in\mathbb{M}_n(\Cpx)_+$,
there exist $y_1$ and $y_2$
such that \eqref{eq:xy1}--\eqref{eq:axay}
hold.
Let $K>1$ be such that
\begin{equation}
\|x_j\|\le K,\qquad\|x_j^{-1}\|\le K.
\end{equation}
Let $E_n$ be the set of all elements of $\mathbb{M}_n(\Cpx)_{sa}$ having spectra in the interval $[K,K^2]$.
We will show that the condition appearing in Lemma~\ref{lem:YI} is satisfied for these sets.
Let $I=\{1,2,\ldots,m\}$ and for every $i\in I$ let $n(i)\in\Nats$,
and $a_1^i,a_2^i\in E_{n(i)}$.
We will find $y_1,y_2\in R^\omega$ such that~\eqref{eq:x1y1}--\eqref{eq:aiyaix} hold.
For any $j\in\{1,2\}$ and $i\in I$, the spectrum of $a_j^i\otimes x_j$
lies in the interval
\begin{equation}
[1,K^3].
\end{equation}
Let $N=\sum_{i=1}^mn(i)$ and let $a_1,a_2\in\mathbb{M}_{N}(\Cpx)$ be the block diagonal matrices
\begin{equation}
a_j=\oplus_{i=1}^mK^{4i}a_j^i,\quad(j=1,2).
\end{equation}
By hypothesis, there exists $y_1,y_2\in R^\omega$ such that~\eqref{eq:xy1}--\eqref{eq:axay} hold.
We have
\begin{equation}
a_1\otimes x_1+a_2\otimes x_2=\oplus_{i=1}^mK^{4i}(a_1^i\otimes x_1+a_2^i\otimes x_2)
\end{equation}
and similarly for $a_1\otimes y_1+a_2\otimes y_2$.
Since the spectrum of $a_j^i\otimes x_j$ lies in $[1,K^3]$ for all $j$ and $i$, the spectrum
of $a_1^i\otimes x_1+a_2^i\otimes x_2$ lies in $[2,2K^3]$ as does the spectrum of
$a_1^i\otimes y_1+a_2^i\otimes y_2$.
Since the intervals in the family
$([2K^{4i},2K^{4i+3}])_{i=1}^m$ are pairwise disjoint, it follows that for every $i\in\{1,\ldots,m\}$,
the projections
\begin{gather*}
(0_{n(1)}\oplus\cdots 0_{n(i-1)}\oplus I_{n(i)}\oplus
0_{n(i+1)}\oplus\cdots 0_{n(m)})\otimes 1_{\mathcal{M}} \\
(0_{n(1)}\oplus\cdots 0_{n(i-1)}\oplus I_{n(i)}\oplus
0_{n(i+1)}\oplus\cdots 0_{n(m)})\otimes 1_{R^\omega}
\end{gather*}
arise as the spectral projection of $a_1\otimes x_1+a_2\otimes x_2$
and, respectively, $a_1\otimes y_1+a_2\otimes y_2$, for the inverval $[2K^{4i},2K^{4i+3}]$.
Cutting by these spectral projections, we thus obtain that the distributions of
$a_1^i\otimes x_1+a_2^i\otimes x_2$ and $a_1^i\otimes y_1+a_2^i\otimes y_2$
are the same, as required.
\end{proof}

\section{Quantum Horn boddies}
\label{sec:QHB}

Let $\mathbb{R}^N_{\ge}$ denote the set of $N$--tuples
of real numbers listed in nonincreasing order.
The {\em eigenvalue sequence} of an $N\times N$ self--adjoint matrix is its sequence of eigenvalues
repeated according to multiplicity and in nonincreasing order, so as to lie in $\mathbb{R}^N_\ge$.
Consider $\alpha=(\alpha_1,\ldots,\alpha_N)$ and $\beta=(\beta_1,\ldots,\beta_N)$
in $\mathbb{R}^N_{\ge}$.
Let $S_{\alpha,\beta}$ be the set of all possible
eigenvalue sequences $\gamma=(\gamma_1,\ldots,\gamma_N)$ of $A+B$, where $A$ and $B$
are self--adjoint $N\times N$ matrices with eigenvalue sequences $\alpha$ and $\beta$, respectively.
Thus,  $S_{\alpha,\beta}$ is the set of all eigenvalue sequences of $N\times N$--matrices of the form
\begin{equation}\label{eq:UV}
U\mathrm{diag}(\alpha)U^*+V\mathrm{diag}(\beta)V^*,\qquad(U,V\in\mathbb{U}_N),
\end{equation}
where $\mathbb{U}_N$ is the group of $N\times N$--unitary matrices.
Klyatchko, Totaro, Knutson and Tao described the set $S_{\alpha,\beta}$
in terms first conjectured by Horn.
See Fulton's exposition~\cite{F00}.
Taking traces, clearly every $\gamma\in S_{\alpha,\beta}$ must satisfy
\begin{equation}\label{eq:tracesum}
\sum_{k=1}^N\gamma_k=\sum_{i=1}^N\alpha_i+\sum_{j=1}^N\beta_j.
\end{equation}
Consider the inequality
\begin{equation}\label{eq:IJK}
\sum_{i\in I}\alpha_i+\sum_{j\in J}\beta_j\ge\sum_{k\in K}\gamma_k.
\end{equation}
for a triple $(I,J,K)$ of subsets of $\{1,\ldots,N\}$.
Horn defined sets $T^n_r$ of triples $(I , J, K )$ of subsets of 
$\{1,\ldots,n\}$ of the same 
cardinality $r$, by the following recursive procedure.
Set 
\begin{equation}
U^n_r = \bigg\{(I , J, K )\bigg|\sum_{i\in I}i+\sum_{j\in J}j=\sum_{k\in K}k+\frac{r(r+1)}2\bigg\}.
\end{equation}
When $r = 1$, set  $T^n_1 = U^n_1$.
Otherwise, let
\begin{equation}
\begin{aligned}
T^n_r = \bigg\{(I,J,K)\in U^n_r\bigg|
    \sum_{f\in F}i_f+\sum_{g\in G}j_g\leq\sum_{h\in H}k_h+\frac{p(p+1)}2,& \\
\text{ for all }p < r\text{ and }(F,G,H)\in T^r_p\;\;\;&\bigg\}.
\end{aligned}
\end{equation}
The result of Klyatchko, Totaro, Knutson and Tao is that $S_{\alpha,\beta}$ consists
of those elements $\gamma\in\mathbb{R}^N_{\ge}$
such that the equality~\eqref{eq:tracesum} holds and the inequality~\eqref{eq:IJK} holds for every triple
$(I,J,K)\in\bigcup_{r=1}^{N-1}T^N_r$.
We will refer to $S_{\alpha,\beta}$ as the {\em Horn body} of $\alpha$ and $\beta$.
It is, thus, a closed, convex subset of $\mathbb{R}^N_\ge$.

The analogue of this situation occuring in finite von Neumann algebras has been considered
by Bercovici and Li~\cite{BL01}, \cite{BL06}; let us summarize part of what they have done.
We denote by $\mathcal{F}$ the set of all right--continuous, nonincreasing, bounded functions
$\lambda:[0,1)\to\mathbb{R}$.
Let $\mathcal{M}$ be a von Neumann algebra with normal, faithful, tracial state $\tau$ and let $a=a^*\in\mathcal{M}$.
The {\em distribution} of $a$ is the Borel measure $\mu_a$, supported on the spectrum of $a$,
such that 
\begin{equation}
\tau(a^n)=\int_{\mathbb{R}}t^n\,d\mu_a(t)\qquad(n\ge1).
\end{equation}
The {\em eigenvalue function} of $a$ is 
$\lambda_a\in\mathcal{F}$ defined by
\begin{equation}
\lambda_a(t)=\sup\{x\in\mathbb{R}\mid\mu_a((x,\infty))>t\}.
\end{equation}
We call $\mathcal{F}$ the set of all eigenvalue functions.
It is an affine space, where we take scalar multiples and sums of functions in the usual way.
Identifying $\mathcal{F}$ with the set of all compactly supported Borel measures on the real line,
it is a subspace of the dual of $C(\mathbb{R})$.
We endow $\mathcal{F}$ with the weak$^*$--topology inherited from this pairing.

It is clear that for every $\lambda\in\mathcal{F}$
and every II$_1$--factor $\mathcal{M}$, there is $a=a^*\in\mathcal{M}$
such that $\lambda_a=\lambda$.
Note that if $\mathcal{M}=M_N(\mathbb{C})$ and if $a=a^*\in M_N(\mathbb{C})$ has eigenvalue
sequence $\alpha=(\alpha_1,\ldots,\alpha_N)$, then its eigenvalue function is given by
\begin{equation}\label{eq:lambda}
\lambda_a(t)=\alpha_j,\quad\frac{j-1}N\le t<\frac jN,\qquad (1\le j\le N).
\end{equation}
In this way, $\mathbb{R}^N_\ge$ is embedded as a subset $\mathcal{F}^{(N)}$ of $\mathcal{F}$,
and the affine structure on $\mathcal{F}^{(N)}$ inherited from $\mathcal{F}$ corresponds
to the usual one on $\mathbb{R}^N_\ge$ coming from the vector space structure of $\mathbb{R}^N$.

For a set $(I,J,K)\in T^N_r$, consider the triple $(\sigma_I^N,\sigma_J^N,\sigma_K^N)$,
where for $F\subseteq\{1,2,\ldots,N\}$, we set
\begin{equation}
\sigma_F^N=\bigcup_{i\in F}\,\bigg[\frac{i-1}N,\frac iN\bigg).
\end{equation}
Let 
\begin{equation}
\mathcal{T}=\bigcup_{N=1}^\infty
\bigcup_{r=1}^{N-1}\{(\sigma_I^N,\sigma_J^N,\sigma_K^N)\mid (I,J,K)\in T^N_r\}.
\end{equation}

\begin{theorem} [\cite{BL06}, Thm. 3.2] \label{thm:BL}
For any $u,v,w\in\mathcal{F}$, there exists
self--adjoint elements $a$ and $b$ in the ultrapower $R^\omega$ of the hyperfinite II$_1$--factor with
$u=\lambda_a$, $v=\lambda_b$ and $w=\lambda_{a+b}$ if and only if
\begin{equation}\label{eq:Tref}
\int_0^1u(t)\,dt+\int_0^1v(t)\,dt=\int_0^1w(t)\,dt
\end{equation}
and, for every $(\omega_1,\omega_2,\omega_3)\in\mathcal{T}$, we have
\begin{equation}\label{eq:Hornef}
\int_{\omega_1}u(t)\,dt+\int_{\omega_2}v(t)\,dt\ge\int_{\omega_3}w(t)\,dt.
\end{equation}
\end{theorem}

Given eigenvalue functions $u,v\in\mathcal{F}$, let $F_{u,v}$ be the set of all $w\in\mathcal{F}$
such that~\eqref{eq:Tref} holds and~\eqref{eq:Hornef} holds
for every $(\omega_1,\omega_2,\omega_3)\in\mathcal{T}$.
Since the functions in $F_{u,v}$ are uniformly bounded,
we see that $F_{u,v}$ is a compact, convex subset of $\mathcal{F}$.

Now we consider an alternative formulation of a special case of Theorem~\ref{thm:BL}.
Let $N\in\mathbb{N}$ and $\alpha,\beta\in\mathbb{R}^N_\ge$.
For $d\in\mathbb{N}$, let
\begin{equation}
K_{\alpha,\beta,d}=\{\lambda_C\mid C=\mathrm{diag}(\alpha)\otimes 1_d+U(\mathrm{diag}(\beta)\otimes1_d)U^*,\,U\in\mathbb{U}_{Nd}\}.
\end{equation}
For $d=1$, this is just the set of eigenvalue functions corresponding to the Horn body $S_{\alpha,\beta}$.
Let
\begin{equation}
K_{\alpha,\beta,\infty}=\overline{\bigcup_{d\ge1}K_{\alpha,\beta,d}}\,.
\end{equation}
As a consequence of Bercovici and Li's results we have the following.
\begin{proposition}
Let $\alpha,\beta\in\mathbb{R}^N_\ge$ and let $u=\lambda_{\mathrm{diag}(\alpha)}$ and 
$v=\lambda_{\mathrm{diag}(\beta)}$ be the correspoding eigenvalue functions.
Then 
\begin{equation}\label{eq:KF}
K_{\alpha,\beta,\infty}=F_{u,v}
\end{equation}
is a compact, convex subset of $\mathcal{F}$.

If Connes' embedding problem has a positive solution, then for every II$_1$--factor $\mathcal{M}$
and every $a,b\in\mathcal{M}_{s.a.}$ whose eigenvalue functions are $u$ and $v$, respectively,
we have $\lambda_{a+b}\in K_{\alpha,\beta,\infty}$.
\end{proposition}
\begin{proof}
The inclusion $\subseteq$ in~\eqref{eq:KF} is clear.
For the reverse inclusion, let $w\in F_{u,v}$.
Then~\eqref{eq:Tref} holds and~\eqref{eq:Hornef} holds for every $(\omega_1,\omega_2,\omega_3)\in\mathcal{T}$.
For $n\in\Nats$, let $w^{(n)}\in\mathcal{F}$ be obtained by averaging over the intervals of length $1/n$, namely,
\begin{equation}
w^{(n)}(t)=\int_{(i-1)/n}^{i/n}f(s)\,ds,\qquad(\frac{i-1}n\le t<\frac in,\quad i\in\{1,2,\ldots,n\}).
\end{equation}
Then $w^{(n)}$ corresponds to an eigenvalue sequence $\gamma\in\mathbb{R}^n_\ge$.
We have
\begin{equation}
\int_0^1u(t)\,dt+\int_0^1v(t)\,dt=\int_0^1w^{(n)}(t)\,dt
\end{equation}
and, for every $(\omega_1,\omega_2,\omega_3)=(\sigma_I^n,\sigma_J^n,\sigma_K^n)\in\mathcal{T}$
for $(I,J,K)\in T^n_r$, we have
\begin{equation}
\int_{\omega_1}u(t)\,dt+\int_{\omega_2}v(t)\,dt\ge\int_{\omega_3}w^{(n)}(t)\,dt.
\end{equation}
Therefore, taking $n=Nd$ to be a multiple of $N$, by the theorem formerly known as Horn's conjecture,
we have $\gamma\in S_{\alpha\otimes1_d,\beta\otimes1_d}$ and, consequently, $w^{(Nd)}\in K_{\alpha,\beta,d}$.
Since $w^{(Nd)}$ converges as $d\to\infty$ to $w$, we have $w\in K_{\alpha,\beta,\infty}$.
This proves the equality~\eqref{eq:KF}.

The final statment is a consequence of Bercovici and Li's result, Theorem~\ref{thm:BL}.
\end{proof}


Bercovici and Li's results provide a means of trying to find a II$_1$--factor $\mathcal{M}$
that lack's Connes' embedding property:
namely, by
finding self--adjoint elements $a,b\in\mathcal{M}$ such that $\lambda_{a+b}\notin F_{\lambda_a,\lambda_b}$;
this amounts to finding
some $(I,J,K)\in T^N_r$
such that 
\begin{equation}
\int_{\sigma^N_I}\lambda_a(t)\,dt+\int_{\sigma^N_J}\lambda_b(t)\,dt<\int_{\sigma^N_K}\lambda_{a+b}(t)\,dt.
\end{equation}

\medskip

On the other hand we will use Theorem~\ref{switch-indices} to see that Connes' embedding problem is
equivalent to an anlogous question about versions of the Horn body with ``matrix coefficients.''

Let $a_1,a_2\in\mathbb{M}_n(\mathbb{C})_{sa}$, and $\alpha,\beta\in\mathbb{R}^N_\ge$.
We introduce the set $K^{a_1,a_2}_{\alpha,\beta}$ of the eigenvalue functions of all
matrices of the form
\begin{equation}\label{eq:aUDU}
a_1\otimes U\mathrm{diag}(\alpha)U^*+a_2\otimes V\mathrm{diag}(\beta)V^*,\qquad(U,V\in\mathbb{U}_N).
\end{equation}
Although, for reasons that will be immediately apparent,
we choose to view $K^{a_1,a_2}_{\alpha,\beta}$ as a subset of $\mathcal{F}$,
we may equally well consider the corresponding eigenvalue sequences
and view $K^{a_1,a_2}_{\alpha,\beta}$ as a subset of $\mathbb{R}^{nN}_\ge$.
Comparing to~\eqref{eq:UV}, the set $K^{a_1,a_2}_{\alpha,\beta}$ is seen to be
the analogue of the Horn body $S_{\alpha,\beta}$, but with ``coefficients'' $a_1$ and $a_2$.
We will refer to these sets as {\em quantum Horn bodies}.

The example below shows that
$K^{a_1,a_2}_{\alpha,\beta}$
need not be convex, even in the case
where $a_1,a_2$ commute.
\begin{example}\label{ex:K}
Let
\begin{equation}
a_1=\left(\begin{matrix}1&0\\0&4\end{matrix}\right),\qquad
a_2=\left(\begin{matrix}2&0\\0&1\end{matrix}\right)
\end{equation}
and let $\alpha=\beta=(2,1)$.
Then the $4\times 4$ matrices of the form~\eqref{eq:aUDU} are all unitary conjugates of the matrices
\begin{equation}
R_t=\left(\begin{matrix}1&0\\0&4\end{matrix}\right)\otimes\left(\begin{matrix}1&0\\0&2\end{matrix}\right)
+\left(\begin{matrix}2&0\\0&1\end{matrix}\right)\otimes
\left(\begin{matrix}1+t&\sqrt{t(1-t)}\\\sqrt{t(1-t)}&2-t\end{matrix}\right),
\end{equation}
for $0\le t\le 1$.
One easily finds the eigenvalues
$\lambda_1(t)\ge\lambda_2(t)\ge\lambda_3(t)\ge\lambda_4(t)$
of $R_t$ to be
\begin{align}
\lambda_1(t)&=\frac{15}2+\frac12\sqrt{25-16t} \\[1ex]
\lambda_2(t)&=\begin{cases}\frac92+\frac12\sqrt{9-8t},&0\le t\le t_1 \\
\frac{15}2-\frac12\sqrt{25-16t},&t_1\le t\le1
\end{cases} \\[1ex]
\lambda_3(t)&=\begin{cases}\frac{15}2-\frac12\sqrt{25-16t},&0\le t\le t_1 \\
\frac92+\frac12\sqrt{9-8t},&t_1\le t\le1
\end{cases} \\[1ex]
\lambda_4(t)&=\frac92-\frac12\sqrt{9-8t},
\end{align}
where $t_1=\frac32\sqrt{65}-\frac{23}2\approx0.593$.
Then the set $\{(\lambda_1(t),\ldots,\lambda_4(t))\mid0\le t\le 1\}$ is a $1$--dimensional subset of $4$--space
that is far from being convex.
For example, a plot of the projection of this set onto the last two coordinates is the curve
in Figure~\ref{fig:lambdaplot}.
\begin{figure}[bht]
\caption{A parametric plot of $\lambda_4$ (vertical axis) and $\lambda_3$ (horizontal axis).}
\label{fig:lambdaplot}
\begin{center}
\epsfig{file=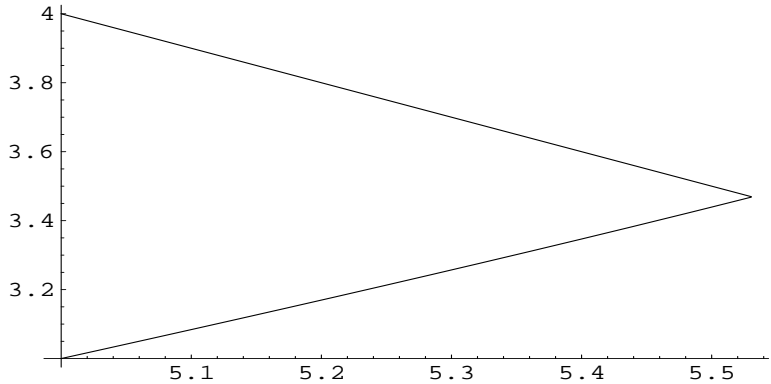,height=2.0in,width=4in}
\end{center}
\end{figure}
The upper part of this curve is a line segment, while the lower part is not.
\end{example}

\smallskip

Extending the notions introduced above,
for integers $d\ge1$, 
let $K^{a_1,a_2}_{\alpha,\beta,d}$ be the set of
the eigenvalue functions of all matrices of the form
\begin{equation}\label{eq:AUDdU}
a_1\otimes U(\mathrm{diag}(\alpha)\otimes 1_d)U^*+a_2\otimes V(\mathrm{diag}(\beta)\otimes 1_d)V^*,
\qquad(U,V\in\mathbb{U}_{Nd}).
\end{equation}
If $d'$ divides $d$, then we have
\begin{equation}
K^{a_1,a_2}_{\alpha,\beta,d'}\subseteq K^{a_1,a_2}_{\alpha,\beta,d}\;.
\end{equation}
Let us define
\begin{equation}
K^{a_1,a_2}_{\alpha,\beta,\infty}=\overline{\bigcup_{d\in\mathbb{N}}K^{a_1,a_2}_{\alpha,\beta,d}}\;,
\end{equation}
where the closure is in the weak$^*$--topology for $\mathcal{F}$ described earlier in this section.
Note that the set $K^{a_1,a_2}_{\alpha,\beta,\infty}$ is compact.

\begin{question}
Though Example~\ref{ex:K} shows that $K^{a_1,a_2}_{\alpha,\beta}$ need not be convex,
is it true that $K^{a_1,a_2}_{\alpha,\beta,\infty}$ must be convex,
or even that $K^{a_1,a_2}_{\alpha,\beta,d}$ must be convex for all $d$ sufficiently large?
Note that it is clear that $K^{a_1,a_2}_{\alpha,\beta,\infty}$ is convex with respect to the affine
structure on $\mathcal{F}$ that arises from taking convex combinations of measures,
under the correspondence between $\mathcal{F}$ and the
set of Borel probability measures on $\mathbb{R}$.
However, we are interested in the other affine structure of $\mathcal{F}$,
resulting from addition of functions on $[0,1)$.
\end{question}

For $a_1,a_2\in\mathbb{M}_n(\mathbb{C})_{s.a.}$ with eigenvalue sequences $\gamma_1,\gamma_2\in\mathbb{R}^n_\ge$,
we obviously have
\begin{equation}
K^{a_1,a_2}_{\alpha,\beta}\subseteq K_{\gamma_1\otimes \alpha,\gamma_2\otimes \beta}
\end{equation}
and
\begin{equation}
K^{a_1,a_2}_{\alpha,\beta,\infty}\subseteq K_{\gamma_1\otimes \alpha,\gamma_2\otimes \beta,\infty}\,.
\end{equation}
The following example shows that these inclusions can be strict.

\begin{example}
Let
\begin{equation}
\begin{aligned}
a_1=\left(\begin{matrix}1&0\\0&0\end{matrix}\right),\qquad
a_2=\left(\begin{matrix}0&0\\0&1\end{matrix}\right).
\end{aligned}
\end{equation}
One directly sees that for any eigenvalue sequences $\alpha$ and $\beta$ of length $N$ and any
$U,V\in\mathbb{U}_N$, the eigenvalue sequence of
\begin{equation}
a_1\otimes U(\mathrm{diag}(\alpha)\otimes 1_d)U^*+a_2\otimes V(\mathrm{diag}(\beta)\otimes 1_d)V^*
\end{equation}
is the re--ordering of the concatenation of $\alpha$ and $\beta$.
Thus, $K_{\alpha,\beta}^{a_1,a_2}$ has only one element.
Moreover, dilating $\alpha$ to $\alpha\otimes 1_d$ does not change the corresponding eigenvalue
functions of 
\begin{equation}
a_1\otimes U\mathrm{diag}(\alpha\otimes1_d)U^*+a_2\otimes V\mathrm{diag}(\alpha\otimes1_d)V^*.
\end{equation}
This shows that $K^{a_1,a_2}_{\alpha,\alpha,\infty}$ has only one element.
Now we easily get
\begin{equation}
K^{a_1,a_2}_{\alpha,\beta,\infty}\ne K_{\alpha\oplus0_N,\beta\oplus0_N},
\end{equation}
where $\alpha\oplus0_N$ means the eigenvalue sequence of $a_1\otimes\mathrm{diag}(\alpha)$, etc.
\end{example}

For $\mathcal{M}$ a II$_1$--factor, we define
$L^{a_1,a_2}_{\alpha,\beta,\mathcal{M}}$ to be the set of all eigenvalue functions of all operators of the form
\begin{equation}
a_1\otimes x_1+a_2\otimes x_2\in\mathbb{M}_n(\mathbb{C})\otimes\mathcal{M},
\end{equation}
where $x_1$ and $x_2$ are self--adjoint elements of $\mathcal{M}$ whose eigenvalue functions
agree with those of the matrices $\mathrm{diag}(\alpha)$ and $\mathrm{diag}(\beta)$, respectively
(see~\eqref{eq:lambda} for an explicit description of the latter).
It is easliy seen that we have
\begin{equation}\label{eq:KLRomega}
K^{a_1,a_2}_{\alpha,\beta,\infty}=L^{a_1,a_2}_{\alpha,\beta,R^\omega}\;.
\end{equation}
Let
\begin{equation}
L^{a_1,a_2}_{\alpha,\beta}=\bigcup_{\mathcal{M}}L^{a_1,a_2}_{\alpha,\beta,\mathcal{M}}\;,
\end{equation}
where the union is over all II$_1$--factors $\mathcal{M}$ with separable predual (acting on a specific separable Hilbert space,
say).
Using an ultraproduct argument, one can show that $L^{a_1,a_2}_{\alpha,\beta}$ is closed in $\mathcal{F}$
and compact.
Also, one obviously has
\begin{equation}
K^{a_1,a_2}_{\alpha,\beta,\infty}\subseteq L^{a_1,a_2}_{\alpha,\beta}.
\end{equation}

Theorem \ref{switch-indices}
gives us the following equivalent formulation of the embedding question.
\begin{theorem}\label{thm:CEP}
The following are equivalent:
\renewcommand{\labelenumi}{(\roman{enumi})}
\begin{enumerate}
\item Every II$_1$--factor $\mathcal{M}$ with separable predual has Connes' embedding property.
\item For all integers $n,N\ge1$ and all
$a_1,a_2\in\mathbb{M}_n(\mathbb{C})_{sa}$, and $\alpha,\beta\in\mathbb{R}^N_\ge$, we have
\begin{equation}\label{eq:KL}
K^{a_1,a_2}_{\alpha,\beta,\infty}= L^{a_1,a_2}_{\alpha,\beta}.
\end{equation}
\end{enumerate}
\end{theorem}
\begin{proof}
Clearly, (i) implies $L^{a_1,a_2}_{\alpha,\beta}=L^{a_1,a_2}_{\alpha,\beta,R^\omega}$, and then from~\eqref{eq:KLRomega}
we get~\eqref{eq:KL}.

Suppose (ii) holds.
It is well known that to solve Connes' embedding problem in the affirmative,
it will suffice to show that every tracial von Neuman algebra $\mathcal{M}$ that is generated by two
self--adjoints $x_1$ and $x_2$ is embeddable in $R^\omega$.

So suppose $\mathcal{M}$ is generated by self--adjoints $x_1$ and $x_2$.
By Proposition~\ref{prop:microstates}, it will suffice to show that $x_1$ and $x_2$ have matricial microstates.
Approximating $x_1$ and $x_2$, if necessary, we may without loss of generality assume that the eigenvalue
functions of both belong to $\mathcal{F}^{(N)}$ for some $N\in\mathbb{N}$, namely, that they correspond
to sequences $\alpha$ and, respectively, $\beta$ in $\mathbb{R}^N_\ge$.
By adding constants, if necessary, we may without loss of generality assume that $x_1$ and $x_2$ are positive
and invertible.
Let $n\in\mathbb{N}$ and let $a_1,a_2\in\mathbb{M}_n(\mathbb{C})$.
Using~\eqref{eq:KLRomega} and~\eqref{eq:KL}, there are $y_1,y_2\in R^\omega$ 
such that~\eqref{eq:xy1}--\eqref{eq:axay} of Theorem~\ref{switch-indices} hold.
So by that theorem, the pair $x_1,x_2$ has matricial microstates.
\end{proof}

\noindent
{\bf Note:} We recently learned of a result of Mika\"el De La Salle~\cite{DLS}
that seems related to our Lemma~\ref{lem:stepII}.

\end{document}